\documentclass{amsart}
\usepackage[numbers]{natbib}
\usepackage{microtype} 

\usepackage{amssymb}
\usepackage{amsmath}
\usepackage{amsthm}
\usepackage{braket}
\usepackage{adjustbox}

\usepackage[shortlabels]{enumitem}

\usepackage{tikz-cd}
\tikzset{
    labl/.style={anchor=south, rotate=90, inner sep=.5mm}
}
\usetikzlibrary{decorations.markings}
\usetikzlibrary{intersections}
\usepackage{pgf,tikz}
\usetikzlibrary{arrows}

\definecolor{qqqqff}{rgb}{0,0,1}

\makeatletter
\tikzcdset{
  open/.code     = {\tikzcdset{hook, circled};},
  closed/.code   = {\tikzcdset{hook, slashed};},
  open'/.code    = {\tikzcdset{hook', circled};},
  closed'/.code  = {\tikzcdset{hook', slashed};},
  circled/.code  = {\tikzcdset{markwith = {\draw (0,0) circle (.375ex);}};},
  slashed/.code  = {\tikzcdset{markwith = {\draw[-] (-.4ex,-.4ex) -- (.4ex,.4ex);}};},
  markwith/.code ={
    \pgfutil@ifundefined%
    {tikz@library@decorations.markings@loaded}%
    {\pgfutil@packageerror{tikz-cd}{You need to say %
      \string\usetikzlibrary{decorations.markings} to use arrows with markings}{}}{}%
    \pgfkeysalso{/tikz/postaction = {
      /tikz/decorate,
      /tikz/decoration={markings, mark = at position 0.5 with {#1}}}
    }
  },
}
\makeatother
\usepackage[english]{babel}
\usepackage{csquotes}
\usepackage{verbatim}

\usepackage[all]{xy}

\usepackage{multicol}

\usepackage{lmodern}
\usepackage[T1]{fontenc}

\usepackage{hyperref}
\hypersetup{breaklinks,colorlinks}
\newtheorem{theorem}{Theorem}[section]

\newtheorem{lemma}[theorem]{Lemma}

\theoremstyle{definition}
\newtheorem{definition}[theorem]{Definition}

\newtheorem{remark}[theorem]{Remark}

\numberwithin{equation}{subsection}

\def\lim{\mathop{\mathrm{lim}}\nolimits}

\def\Pic{\mathop{mathrm{Pic}}}

\def\id{\mathop{\mathrm{id}}\nolimits}

\def\Pic{\mathop{\mathrm{Pic}}\nolimits}

\def\deg{\mathop{\mathrm{deg}}\nolimits}
\def\max{\mathop{\mathrm{max}}\nolimits}

\def\det{\mathop{\mathrm{det}}\nolimits}

\def\dim{\mathop{\mathrm{dim}}\nolimits}

\def\lim{\mathop{\mathrm{lim}}\nolimits}

\def\im{\mathop{\mathrm{im}}\nolimits}

\newtheorem{thmx}{Theorem}

\title[Stratifying by decomposition type]{Stratifying the moduli space of stable vector bundles by decomposition type}
\author{Dario Weissmann}
\address[D.~Weissmann]{Fakult\"{a}t f\"{u}r Mathematik \\
Universit\"{a}t Duisburg-Essen \\
Universit\"{a}ts-strasse 2 \\
45141 Essen \\
Germany}
\email{\href{mailto: dario.weissmann@uni-due.de}{dario.weissmann@uni-due.de}}
\date{\today}

\begin{document}

\begin{abstract}
    The moduli space of $\mu$-stable vector bundles on a normal projective variety
    over an algebraically closed field of characteristic $p\geq 0$
    is stratified with respect to the decomposition type.
    On a smooth projective curve of genus at least $2$
    we obtain mostly sharp dimension estimates for these strata.

    As an application, we obtain a dimension estimate 
    for the closure of the prime to $p$
    trivializable stable bundles in the moduli space of stable vector bundles.
\end{abstract}

\maketitle

\section{Introduction}

Consider a smooth projective curve $C$ over an algebraically closed field
of characteristic $p\geq 0$.
Recently stability of vector bundles has been investigated under pullback
by \'etale Galois covers $D\to C$ of degree prime to the characteristic $p$, see \cite{fun}.
In this paper we extend upon the results of \cite{fun} to obtain
a stratification of the moduli space and estimate the dimension of these
strata over a smooth projective curve.
As an application, we study the closure of the prime to $p$ trivializable vector bundles.

Consider a stable vector bundle $V$ on $C$.
Given an \'etale Galois cover $\pi:D\to C$
the pullback of $V$ to $D$ may fail to remain stable.
However, $V$ is still polystable on $D$ 
and can only decompose in a very special way:
It is a direct sum of stable vector bundles $W_i$ such that 
the Galois group acts transitively on the isomorphism classes of the $W_i$,
see \cite[Lemma 3.8]{fun}.
In particular, the $W_i$ have the same rank $m$.
Fixing the rank $m$ of the $W_i$ as well as the rank $r$ and degree $d$ of $V$,
we obtain a stratification $Z^{s,r,d}(m,\pi)$ of the moduli space 
of stable vector bundles $M^{s,r,d}_C$ on $C$
with respect to a Galois cover
- the \emph{decomposition stratification} with respect 
to the \'etale Galois cover $\pi:D\to C$.
We can estimate the dimension of these strata generalizing
 \cite[Lemma 4.4]{fun}:

\begin{thmx}[Theorem \ref{theorem-dimension-estimate} (ii)]
    \label{thmx-dimension-estimate}
	Let $\pi:D\to C$ be an \'etale Galois cover of a smooth projective curve $C$ of
	genus $g_C \geq 2$.
	Let $r\geq 2$ and $d\in \mathbf{Z}$.
	Let $m\mid r$.
	Then we have
 \[
       \dim(Z^{s,r,d}(m,\pi))\leq rm(g_C-1)+1.
\]
\end{thmx}

While one could also define the decomposition stratification with respect
to a Galois morphism instead of a \'etale Galois cover, this does
not yield a different stratification. Indeed, in \cite[Theorem 1.2]{bdp} it was shown by Biswas, Das, and Parameswaran that genuinely ramified morphisms of normal projective varieties preserve stability under pullback.
Thus, only the \'etale part of a Galois morphism matters for the decomposition
behaviour.

Note that this stratification is dependent on the choice of the cover $D\to C$.
To obtain a canonical stratification, recall that
a vector bundle is \emph{prime to $p$ stable} if it remains stable
on all prime to $p$ Galois covers, see \cite[Definition 2.2]{fun}.
There is a prime to $p$ cover $C_{r,good}\to C$ checking for prime to $p$ stability of vector bundles of rank $r$,
see \cite[Theorem 1]{fun}.
Iterating the cover $C_{r,good}$, we obtain a prime to $p$ Galois cover
$C_{r,split}\to C$ checking for the eventual decomposition 
with respect to prime to $p$ Galois covers. 

The decomposition stratification with respect to
$C_{r,split}\to C$ is independent of the cover $C_{r,split}\to C$
and is called the \emph{prime to $p$ decomposition stratification},
see Definition \ref{definition-prime-to-p-strata}.
Fixing the degree $d$,
we denote this stratification for $m\mid r$
by $Z^{s,r,d}(m)$.
Then the dimension estimates obtained for arbitrary Galois covers are mostly sharp.
This is a generalization of \cite[Theorem 2]{fun}.

\begin{thmx}[Theorem \ref{theorem-dimension-estimate-sharp}]
    \label{thmx-dimension-estimate-sharp}
    Let $C$ be a smooth projective curve of genus $g_C\geq 2$.
    Let $r\geq 2$ and $d\in \mathbb{Z}$.
    Then for $m\mid r$ we have the following:
    \begin{itemize}
        \item $\dim(Z^{s,r,d}(m))\leq (\frac{r}{m})'m^2(g_C-1)+1$, where $(\frac{r}{m})'$ denotes the prime to $p$ part of $\frac{r}{m}$.
        \item If $p \nmid \frac{r}{m}$, then we have $\dim(Z^{s,r,d}(m))=rm(g_C-1)+1$.
    \end{itemize}
\end{thmx}

As an application, we can estimate the dimension of the closure of the prime to $p$ trivializable (semi)stable vector bundles.
This should be compared on the one hand with a result of Ducrohet and Mehta,
see \cite[Corollary 5.1]{dm},
asserting that in positive characteristic the \'etale trivializable bundles are dense in the moduli space
of semistable vector bundles of degree $0$.
On the other hand, the non-density of the prime to $p$ trivializable bundles
was recently obtained in \cite[Corollary 3]{fun} and in rank $2$ and characteristic $0$
by Ghiasabadi and Reppen, see \cite[Corollary 4.16]{gr}. Thus, it is natural to ask for a description
of the closure of the prime to $p$ trivializable stable vector bundles.
This is closely related to the smallest prime to $p$ decomposition stratum:

\begin{thmx}[Theorem \ref{theorem-closure}]
\label{thmx-closure}
    Let $C$ be a smooth projective curve of genus $g_C\geq 2$.
    Let $r\geq 2$.
    Let $Z^{s,r}$ (resp. $Z^{ss,r}$) be the closure of the prime to $p$ 
    trivializable (semi)stable vector bundles in $M^{s,r,0}_C$ 
    (resp. $M^{ss,r,0}_C$).
    Then we have the following:
    \begin{itemize}
        \item $Z^{s,r}\subseteq Z^{s,r,0}(1)$.
        \item $\dim(Z^{s,r})\leq r'(g_C-1)+1$,
            where $r'$ is the prime to $p$ part of $r$.
        \item If $p\nmid r$, then $\dim(Z^{s,r})= \dim(Z^{s,r,0}(1))=r(g_C-1)+1$.
        \item $\dim(Z^{ss,r})=rg_C$.
    \end{itemize}
\end{thmx}

The paper is structured as follows:
We begin the first section by recalling some properties of semi\-stable vector bundles under pushforward and pullback. Then we define the decomposition stratification and estimate the dimension of the strata.

In the second section, we study stable vector bundles under \'etale pushforward.
Pushforward under an \'etale morphism of smooth projective curves
induces a finite morphism on the level
of moduli spaces.
Furthermore, we identify the big open
where the direct image of a stable vector bundle is stable.

In the third section, we define the prime to $p$ decomposition stratification and
obtain the mostly sharp dimension estimates.

In the last section, we apply the theory developed to study the closure of the prime to $p$ trivializable (semi)stable vector bundles.

\subsection*{Notation}
We work over an algebraically closed field $k$ of characteristic $p\geq 0$.
	A variety is a separated integral scheme of finite type over $k$.
	A curve is a variety of dimension $1$.
	
	
	If $X$ is a projective variety we implicitly choose an ample bundle
	$\mathcal{O}_X(1)$ on $X$.  
        Given a finite morphism $\pi:Y\to X$ we
	always equip $Y$ with the polarization $\mathcal{O}_Y(1)=\pi^{\ast}\mathcal{O}_X(1)$.
	By (semi)stability we mean $\mu$-(semi)stability
        with respect to $\mathcal{O}_X(1)$.
	
	We denote the moduli space of (semi)stable 
        vector bundles of rank $r$ and degree $d$ on
	a smooth projective 
        curve $C$ by $M^{s,r,d}_C$ (resp. $M^{ss,r,d}_C$).
	On a projective variety $X$ the stable vector bundles with Hilbert             polynomial $P$ 
	form an open $M^{s,P}_X$ in the moduli space of 
        Gieseker semistable sheaves $M^{G-ss,P}_X$.
		
	Given a morphism $\pi:Y\to X$ of varieties and
        a sheaf $F$ on $X$ we denote the pullback
	$\pi^{\ast}F$ also by $F_{\mid Y}$.
	
	By a Galois morphism $Y\to X$ of varieties we mean a finite
        surjective separable morphism
	such that the extension of function fields is Galois.
	A (Galois) cover is a finite \'etale (Galois) morphism $Y\to X$.
        A cyclic Galois cover is a Galois cover with cyclic Galois group.

        We call a Galois cover $\pi:Y\to X$ prime to $p$ if $p\nmid \deg(\pi)$.
\subsection*{Acknowledgements}

The author would like to thank his PhD advisor Georg Hein for his support and many inspiring and fruitful discussions.
He would also like to thank his second advisor Jochen Heinloth for some helpful discussions
and suggesting the notation $X_{r,split}$.
Thanks also go to Stefan Reppen for pointing out the literature \cite{dm} and a discussion on \'etale trivializable bundles.
The author was funded by the DFG Graduiertenkolleg 2553.
This paper is part of the author's PhD thesis carried out at the University of Duisburg-Essen.

\section{Stratifying by decomposition type}

We define a stratification associated 
to the decomposition type of 
a stable vector bundle with respect to a fixed Galois cover $Y\to X$ of a normal projective variety.
We estimate the dimension of the strata in the curve case refining 
the result \cite[Lemma 4.4]{fun}.

For the convenience of the reader we recall some properties 
of semi\-stable vector bundles under pushforward and pullback:

\begin{lemma}[Lemma 3.2.1-3.2.3, \cite{hl}, see also Lemma 3.2 \cite{fun}]
	\label{lemma-basic}
	Let $\pi:Y\to X$ be a finite separable morphism of normal projective
	varieties of degree $d$.
	Let $V$ be a vector bundle on $X$ and $W$ be a vector bundle on $Y$.
        Then we have the following:
	\begin{itemize}
		\item $\mu(W)=d(\mu(\pi_{\ast}W)-\mu(\pi_{\ast}\mathcal{O}_Y))$.
		\item $\mu(V_{\mid Y})=d\mu(V)$.
            \item $V$ is semistable iff $V_{\mid Y}$ is semistable.  
            \item If $V_{\mid Y}$ is stable, then so is $V$.
        \end{itemize}
        If in addition $\pi:Y\to X$ is Galois, then we have:
        \begin{itemize}
		\item If $V$ is polystable, then $V_{\mid Y}$ is polystable.
            \item If $\pi$ is prime to $p$,
			then $V$ is polystable iff $V_{\mid Y}$ is polystable.
	\end{itemize}
\end{lemma}

In fact, more can be said for the pullback of a stable vector bundle under a Galois cover.
The following is the key observation for defining decomposition strata:
\begin{lemma}[Lemma 3.8, \cite{fun}]
\label{lemma-key}
    Let $Y\to X$ be a Galois cover of a normal projective variety $X$.
    Then a stable vector bundle $V$ on $X$ decomposes on $Y$ as a direct sum 
    \[
        V_{\mid Y}\cong \bigoplus_{i=1}^n W_i^{\oplus e}
    \]
    of pairwise non-isomorphic stable vector bundles $W_i$
    such that the Galois group acts transitively on the isomorphism classes of the $W_i$.
\end{lemma}

\begin{definition}
    Let $\pi:Y\to X$ be a Galois cover of a normal projective variety $X$.
    Let $V$ be a stable vector bundle of rank $r$ on $X$.
    The \emph{decomposition type of $V$ with respect to $\pi$}
    is the rank $m$ of the bundles $W_i$ in the decomposition
        $V_{\mid Y}\cong \bigoplus_{i=1}^n W_i^{\oplus e}$
    of Lemma \ref{lemma-key}.
    The \emph{refined decomposition type of $V$ with respect to $\pi$}
    is the tuple $(n,e)$. 
    
    Note that the refined decomposition type recovers the
    decomposition type as $mne=r$.

    Let $P\in\mathbb{Q}[x]$ be a polynomial.
    Assume that the moduli space  $M^{s,P}_X$ of stable vector bundles with Hilbert polynomial $P$
    is non-empty. The Hilbert polynomial determines the rank of the vector 
    bundles in $M^{s,P}_X$ which we denote by $r$.
    
    For integers $m,n,e,\geq 1$ such that $mne=r$ we define the 
    \emph{refined decomposition strata with respect to $\pi$}
    \[
        Z^{s,P}(n,e,\pi) := 
        \left\{ 
            V\in M^{s,P}_X 
        \middle \vert
        \begin{array}{l}
            V \text{ has refined decomposition}\\
            \text{type } (n,e) \text{ with respect to } \pi
        \end{array}
        \right\}, \text{ and}
    \]
    \[
        Z^{s,P}(\cdot \mid n, e\mid \cdot,\pi):=
        \bigsqcup_{(n',e')}Z^{s,P}(n',e',\pi),
    \]
    where the union is taken over $n',e'\geq 1$ such that $n'e'=ne$
    and $n'\mid n$.
    The \emph{decomposition strata with respect to $\pi$} are defined as
    \[
        Z^{s,P}(m,\pi) := 
        \left\{
            V \in M^{s,P}_X 
        \middle \vert
            \begin{array}{l}
                V \text{ has decomposition type } m\\
                \quad \quad \text{ with respect to } \pi
            \end{array} 
        \right\}, \text{ and}
    \]
    \[
        Z^{s,P}(\cdot \mid m,\pi) := \bigsqcup_{m'\mid m}Z^{s,P}(m',\pi).
    \]
    If $X$ is a smooth projective curve, then the Hilbert polynomial is determined by the
    rank $r$ and degree $d$ and we use the notations
    \[
        Z^{s,r,d}(m,\pi), Z^{s,r,d}(\cdot \mid m,\pi), Z^{s,r,d}(n,e,\pi), \text{ and } Z^{s,r,d}(\cdot \mid n, e\mid \cdot,\pi).
    \]
\end{definition}

We show that $Z^{s,P}(m,\pi)$ and $Z^{s,P}(n,e,\pi)$ 
form a stratification:

\begin{lemma}
\label{lemma-strata-closed}
    Let $\pi:Y\to X$ be a Galois cover of a normal variety $X$.
    Let $P\in\mathbb{Q}[x]$ be a polynomial.
    Assume that $M^{s,P}_X$ is non-empty and denote by $r$ the rank of the vector bundles in $M^{s,P}_X$.
    Then we have the following:
    \begin{enumerate}[(i)]
    \item $Z^{s,P}(\cdot \mid m, \pi)\subseteq M^{s,P}_X$ is closed for $m\mid r$.
    \item $Z^{s,P}(\cdot \mid n, e\mid \cdot ,\pi)\subseteq Z^{s,P}(m,\pi)$ 
    is closed for $m,n,e\geq 1$ such that $mne=r$.
    \end{enumerate}
\end{lemma}

\begin{proof}
    As $M^{s,P}_X$ is quasi-projective, it has only finitely many connected components $C_1,\dots,C_l$.
    On each of these components the Hilbert polynomial $P_j:=P(V_{\mid Y})$ of $V\in C_j$ is independent
    of $V$ as the Euler characteristic is locally constant, see \cite[Corollary, p.50]{av}. 
    
    We claim that for $V\in C_j$ the bundles $W_i$ in the direct sum
    decomposition $V_{\mid Y}\cong \bigoplus_{i=1}^n W_i^{\oplus e}$
    of Lemma \ref{lemma-key} have the same Hilbert-polynomial.
    Indeed, the Galois group $G$ of $Y\to X$
    acts transitively on the isomorphism classes of the $W_i$.
    For $\sigma\in G$ we have $P(\sigma^{\ast}W_i)=P(W_i)$
    as the Hilbert-polynomial is computed with respect to
    $\pi^{\ast}\mathcal{O}_X(1)$ which is invariant under the action of $G$.
    Thus, we have
    $P(W_i)=\frac{m}{r}P_j$
    for $V\in Z^{s,P}(m,\pi)\cap C_j$.
    
    Pulling back along $\pi$ defines a morphism $\pi_j^{\ast}:C_j \to M^{G-ss,P_j}_Y$.
    In $M^{G-ss,P_j}_Y$ the set-theoretic image $\im(\varphi_{n,e})$ of 
    \[
        \varphi_{n,e}:\prod_{i=1}^n M^{G-ss, \frac{m}{r}P_j}_Y\to 
        M^{G-ss,P_j}_Y,(W_1,\dots,W_n)\mapsto \bigoplus_{i=1}^n W_i^{\oplus e}
    \]
    is closed as it is a morphism of projective schemes, 
    where $m,n,e\geq 1$ such that $mne=r$.
    
    We claim that $Z^{s,P}(\cdot \mid m,\pi)\cap C_j$ is the preimage of 
    $\im(\varphi_{n,1})$ under $\pi_j^{\ast}$.
    Indeed, $V\in Z^{s,P}(\cdot \mid m,\pi)\cap C_j$ clearly lies in the preimage.
    If $V_{\mid Y}\in \im(\varphi_{n,1})$ for some $V\in C_j$ and $n$ such that $mn=r$,
    then $V_{\mid Y}$ is $S$-equivalent to $\bigoplus_{i=1}^n W_i$ 
    for some Gieseker-semistable sheaves $W_i$ of rank $m$.
    As $V_{\mid Y}$ is a direct sum of stable vector bundles with the 
    same Hilbert polynomial, it is Gieseker-polystable. The associated graded object
    of the JH-filtration is unique and thus
    $V$ has decomposition type $m'$ with respect to $\pi$ 
    for some $m'\mid m$.
    
    Then (i) follows from 
    $Z^{s,P}(\cdot \mid m,\pi)=\bigsqcup_{j=1}^{l}Z^{s,P}(\cdot \mid m,\pi)\cap C_j$.

    Similarly, (ii) is obtained from 
    \[
        (\pi_j^{\ast})^{-1}(\im(\varphi_{n,e}))\cap Z^{s,P}(m,\pi)=Z^{s,P}(\cdot \mid n ,e\mid \cdot,\pi)\cap C_j.
    \]
\end{proof}

\begin{remark}
    Ordering $Z^{s,P}(m,\pi), m\mid r,$ via division, we obtain
    a stratification of $M^{s,P}_X$ by $Z^{s,P}(m,\pi)$.
    Furthermore, the refined decomposition strata $Z^{m,P}(n,e,\pi)$ 
    stratify the
    decomposition stratum $Z^{s,P}(m,\pi)$ for $mne=r$ if we order them via 
    \[
        Z^{s,P}(n', e',\pi)\leq Z^{s,P}(n,e,\pi):\Leftrightarrow n'\mid n.
    \]
\end{remark}

\subsection{Dimension estimates}

Consider a Galois cover $\pi:D\to C$ of a smooth projective curve with $C$ Galois group $G$.
To estimate the dimension of the (refined) 
decomposition strata with respect to $\pi$
we show that the decomposition of the pullback $V_{\mid D}\cong \bigoplus_{i=1}^n W_i^{\oplus e}$ of Lemma \ref{lemma-key}
can essentially be recovered from $W_1$.
Furthermore, $W_1$ behaves for the dimension estimates as if it descends to an intermediate
cover $D'\to C$ of degree $n$.

Consider the case $n=1$.
Then there is only one isomorphism class of the conjugates of $W_1$, i.e., $W_1$ is $G$-invariant.
This does not necessarily mean that $W_1$ descends to $C$. 
However, it does up to a twist by a line bundle, see \cite[Lemma 4.3]{fun}.
The following is a generalization of the estimate obtained in \cite[Lemma 4.4]{fun}
and uses the same techniques in the proof.

\begin{theorem}
    \label{theorem-dimension-estimate}
	Let $\pi:D\to C$ be a Galois cover of a smooth projective curve $C$ of
	genus $g_C \geq 2$.
	Let $r\geq 2$ and $d\in \mathbf{Z}$.
	Let $m,n,e\geq 1$ such that $mne=r$.
        Let $r=r'$ (resp. $m'$) be the prime to $p$ part of $r$ (resp. $m$).
	Then we have the following:
        \begin{enumerate}[(i)]
	    \item $\dim(Z^{s,r,d}(n,e,\pi))\leq nm^2(g_C-1)+1$.
            \item $\dim(Z^{s,r,d}(m,\pi))\leq rm(g_C-1)+1$.
	    \item If $\pi$ is a prime to $p$ cover, then 
                \[
                    \dim(Z^{s,r,d}(m,\pi))\leq \frac{r'}{m'}m^2(g_C-1)+1.
                \]
            \item If $\pi$ is a prime to $p$ cover and $r=p^l,l\geq 1$, then $Z^{s,r,d}(n,e,\pi)$ is empty for $n\geq 2$. 
        \end{enumerate}
\end{theorem}

\begin{proof}
    Let $G$ denote the Galois group of $D\to C$.
    Note that the decomposition strata are locally closed in $M^{ss,r,d}_C$, see Lemma \ref{lemma-strata-closed}. 
    Thus, the dimension of $Z^{s,r,d}(n,e,\pi)$ is the same as the dimension of its closure
    in $M^{ss,r,d}_C$.
    
    (i): Consider $V\in Z^{s,r,d}(n,e,\pi)$. Then
    we have $V_{\mid D}\cong \bigoplus_{i=1}^n W_i^{\oplus e}$ for
    some pairwise non-isomorphic stable vector bundles $W_i$ of rank $m$.
    Let $H$ be the stabilizer of the isomorphism class of $W_1$.
    Then $\pi':D':=D/H\to C$ is an intermediate cover of $D\to C$ of degree $n$ over $C$ as $G$ acts transitively on the isomorphism classes of the $W_i$.

    Fix an inclusion $\iota:W_1^{\oplus e}\to V_{\mid D}$.
    Let $W$ be the image of 
    \[
        \bigoplus_{\sigma\in H}\sigma^{\ast}W_1^{\oplus e}
        \xrightarrow{\oplus \varphi_{\sigma}^{-1}\sigma^{\ast}\iota}V_{\mid D},
    \]
    where $\varphi_{\sigma}:V_{\mid D}\to \sigma^{\ast}V_{\mid D}$ denotes the $G$-linearization associated to $V$.
    Then $W$ is an $H$-invariant saturated subsheaf of $V_{\mid D}$ isomorphic to $W_1^{\oplus e}$. 
    Thus, $W\subseteq V_{\mid D}$ descends to a saturated subsheaf 
    $W'\subseteq V_{\mid D'}$ by \cite[Lemma 3.6]{fun}.

    Since the action of $G$ on the isomorphism classes of the $W_i$ is transitive,
    we obtain that $V_{\mid D}$ is isomorphic to a direct sum of conjugates of $W'_{\mid D}$.
    As $W'_{\mid D}$ is semistable, so is $W'$ by Lemma \ref{lemma-basic}.
    
    By \cite[Lemma 4.3]{fun}, there exists a line bundle
    $L$ on $D$ and a vector bundle $W'_1$ on $D'$ such that
    $(W'_{1})_{\mid D}\cong L\otimes  W_1$.
    Then $W'_1$ is stable.
    Note that $L^{\otimes em}$ descends to $D'$ as
    \[
        \det(W'_1)_{\mid D}\cong L^{\otimes m}\det(W_1) \text{ and }
        \det(W_1)^{\otimes e}\cong \det(W')_{\mid D}.
    \]

    Tensoring $W'_1$ by a line bundle of degree $1$ changes the degree
    of $W'_1$ by $m$ and we can assume that $W'_1$ has degree $d',
    0\leq d'< m$.
    Fixing the degree of $W'_1$ also fixes the degree of $L$.
    Choose a line bundle $L'$ on $D'$ of degree $d', 0\leq d'< m,$
    and denote the moduli space of semistable vector bundles of rank $m$
    and determinant $L'$ on $D'$ by $M^{ss,m}_{L'}$.
    Denote by $P(d')$ the moduli space of line bundles on $D$ of degree 
    $f$ such that their $me$-th power descends to $D'$
    and 
    \[
        rf=d\deg(\pi)-e\deg(\pi)d'.
    \]
   Note that if $P(d')\neq \emptyset$, then $\dim(P(d'))=g_{D'}$ as raising a line bundle 
   to its $me$-th power is a finite morphism $\Pic^f_D \to \Pic^{mef}_D$.
   
    Consider for a finite subset $\Sigma\subseteq G$ of cardinality $n$ the morphism
    \[
        \varphi_{d',\Sigma,D'}:P(d') \times M^{ss,m}_{L'} \to M^{ss, r, d\deg(\pi)}_D,
        (L,W'_1) \mapsto \bigoplus_{\sigma\in \Sigma}\sigma^{\ast}(L\otimes (W'_1)_{\mid D})^{\oplus e}.
    \]
    Observe that 
    the image $Z_{d',\Sigma,D'}$ of $\varphi_{d',\Sigma,D'}$ 
    is closed as $\varphi_{d',\Sigma,D'}$ is a morphism of projective varieties.
    
    The above discussion shows that 
    \[
        \pi^{\ast}(Z^{s,r,d}(n,e,\pi))\subseteq \bigcup_{d'=0}^{m-1}\bigcup_{D'\to C}\bigcup_{\Sigma\subseteq G}Z_{d',\Sigma,D'},
    \]
    where the union is taken over intermediate covers $D\to D'\to C$ of degree $n$ and subsets $\Sigma\subseteq G$ of cardinality $n$.
    
    We can now estimate the dimension of the refined decomposition strata.
    By \cite[Theorem 4.2]{fun}, pullback by $\pi$ is a finite
    morphism $\pi^{\ast}$ and it suffices to estimate $\dim(\pi^{\ast}(Z(n,e,\pi)))$.
    We have
    \[
        \dim(Z_{d',\Sigma,D'})\leq \dim(P(d'))+\dim(M^{ss,m}_{L'})=m^2(g_{D'}-1) + 1.
    \]
    We obtain
    \[
        \dim(Z^{s,r,d}(n,e,\pi))=\dim(\pi^{\ast}(Z^{s,r,d}(n,e,\pi)))\leq m^2n(g_C-1)+1
    \]
    by Riemann-Hurwitz.

    (ii): This is a direct consequence of (i) as
    \[
        Z^{s,r,d}(m,\pi)=\bigsqcup_{(n,e)} Z^{s,r,d}(n,e,\pi),
    \]
    where the union is taken over $n,e\geq 1$ such that $mne=r$.
    Then 
    \[
        \dim(Z^{s,r,d}(m,\pi)) = \max_{n,e}\dim(Z^{s,r,d}(n,e,\pi)) \leq rm(g_C-1)+1
    \]
    as the maximum is obtained at $e=1, n=\frac{r}{m}$.
    
    (iii): If $\pi$ is prime to $p$ and $Z^{s,r,d}(m,\pi)$ is non-empty, then we claim that $n$ is coprime to $p$ as well.
        Indeed, we have seen in (i) that $n\mid \deg(\pi)$.
        Then the maximum in the estimate of (ii)
        is obtained at $e=\frac{rm'}{r'm}, n=\frac{r'}{m'}$.
        Thus, we conclude
        \[
            \dim(Z^{s,r,d}(m,\pi))\leq \frac{r'}{m'}m^2(g_C-1)+1.
        \]

    (iv): If $Z^{s,r,d}(n,e,\pi)$ is non-empty, then
        $n\mid \deg(\pi)$. This is impossible under the assumptions of (iv).
\end{proof}

    \section{Pushforward is finite}
In this section we study the pushforward of stable vector bundles under an \'etale morphism.
We identify the open where the pushforward is stable. Furthermore,
pushforward induces a finite morphism on the level of moduli spaces 
of semistable vector bundles over smooth projective curves.

The following lemma generalizes \cite[Lemma 4.1]{fun}.
    \begin{lemma}
	\label{lemma-etale-pushforward}
	Let $\pi:Y\to X$ be a cover of a normal projective variety $X$ of dimension $\geq 1$.
	Then we have the following:
	\begin{enumerate}[(i)]
            \item $\pi_{\ast}\mathcal{O}_Y$ has slope $0$.
		\item The pushforward $\pi_{\ast}W$ of a 
                semistable vector bundle $W$ on $Y$ is
			semistable of degree $\deg(W)$.
		\item Let $V$ be a semistable vector bundle on $X$.
			Then $\pi_{\ast}\mathcal{O}_Y\otimes V$ is semistable
			of slope $\mu(V)$.
	\end{enumerate}
\end{lemma}

\begin{proof}
        (i): If $X$ is a curve, then $\pi_{\ast}\mathcal{O}_Y$ has degree $0$ by Riemann-Hurzwitz.

        We claim that the general case reduces to the curve case. By Bertini's theorem
        the general complete intersection curve $C$ on $X$
        is irreducible and the same holds for $D:=Y\times_X C$, see \cite[Corollaire 6.11 (3)]{jou}.
        Furthermore, the general such $C$ is normal, see \cite[Theorem 7]{sei}.
        As the projection of $\pi':D=Y\times_X C\to C$
        is \'etale, $D$ is normal as well.
        
        We can compute the degree of $\pi_{\ast}\mathcal{O}_Y$ on $C$.
        As $\pi$ is finite, we can apply affine base change to obtain $(\pi_{\ast}\mathcal{O}_Y)_{\mid C}\cong \pi'_{\ast}\mathcal{O}_D$
       and the result follows.
        
	(ii): This short argument can already be found in the proof of \cite[Proposition 5.1]{bp} 
        for line bundles of degree $1$ on a smooth projective curve.
        
        Let $W$ be a semistable vector bundle of slope $\mu$ and rank $r$ on $Y$.
        By (i) we have $\deg\pi_{\ast}\mathcal{O}_Y=0$.
        Thus,
        the pushforward $\pi_{\ast}W$ has slope $\mu/\deg(\pi)$.
	If $\pi_{\ast}W$ was not semistable, consider the maximal destabilizing
	subsheaf $V$ of $\pi_{\ast}W$. By adjunction $\pi^{\ast}V\to W$ is a
	non-zero morphism of semistable torsion-free sheaves contradicting
	\[
	\mu(\pi^{\ast}V)
	=\deg(\pi)\mu(V)>\deg(\pi)\mu(\pi_{\ast}W)
	=\mu(W).
	\]

	(iii): Let $V$ be a semistable vector bundle on $X$.
	    Using (i) we obtain
	\[
            \mu(V)=\mu(\pi_{\ast}(\mathcal{O}_D))+\mu(V)=\mu(\pi_{\ast}(\mathcal{O}_D)\otimes V).
        \]
	By the projection formula, we have $\pi_{\ast}\mathcal{O}_Y\otimes
	V\cong \pi_{\ast}\pi^{\ast}V$ which is semistable by (ii) and Lemma \ref{lemma-basic}.
\end{proof}

\begin{lemma}
\label{lemma-pushforward-finite}
    Let $\pi:D\to C$ be a cover of degree $n$ of a smooth projective curve $C$ of genus $g_C\geq 2$.
    Let $r\geq 1$ and $d\in\mathbb{Z}$.
    Then pushforward induces a finite morphism
    \[
        \pi_{\ast}:M^{ss,r,d}_D\to M^{ss,rn,d}_C, W\mapsto \pi_{\ast}W.
    \]
\end{lemma}

\begin{proof}
    By Lemma \ref{lemma-etale-pushforward}, the pushforward of a 
    semistable vector bundle on $D$ is semistable on $C$ and of the same degree.
    As pushforward along a finite flat morphism is well-behaved in families,
    we obtain the morphism on the level of moduli spaces.

    To show finiteness it suffices to show quasi-finiteness as $\pi_{\ast}$ is a morphism of projective varieties.
    Consider $V\in M^{ss,rn,d}_C$. For $W$ semistable on $D$ such that $\pi_{\ast}W\cong_S V$
    we find that $\pi^{\ast}\pi_{\ast}W\cong_S \pi^{\ast}V$.
    As $\pi$ is affine, the counit $\pi^{\ast}\pi_{\ast}W\to W$ is surjective. Furthermore, $\pi^{\ast}\pi_{\ast}W$ 
    is semistable of the same slope as $W$. We conclude that $W$ is $S$-equivalent to
    a direct sum of stable vector bundles appearing in the JH-filtration of $\pi^{\ast}V$.
    As the associated graded of the JH-filtration is unique, 
    there are only finitely many possibilities for the $S$-equivalence class of $W$.
    We conclude that $\pi_{\ast}$ is quasi-finite and thus finite.
\end{proof}

\begin{lemma}
\label{lemma-construction-orbit}
    Let $\pi:Y\to X$ be a Galois cover of a normal projective variety $X$ with Galois group $G$.
    Let $P\in\mathbb{Q}[x]$.
    Then for $\sigma\in G\setminus \{e_G\}$
    \[
        U^P_{\sigma}:=\{W\in M^{G-ss,P}_Y \mid \sigma^{\ast}W\ncong_S W\}\subseteq M^{G-ss,P}_Y
    \]
    and $U^P:=\bigcap_{\sigma\in G\setminus\{e_G\}}U_{\sigma}\cap M^{s,P}_X$ are open.
    We have the following:
    \begin{enumerate}[(i)]
       \item For $W\in U^{P}$ the direct image $\pi_{\ast}W$ is
        a stable vector bundle on $X$.
       \item $W\in U^{P}$ iff $W\in M^{s,P}_Y$ and $W$ 
        does not descend to a stable vector bundle $W'$ on an intermediate
        cover $Y\to Y'\to X$ such that $Y'\ncong Y$.
    \end{enumerate}
    On a smooth projective curve, the Hilbert polynomial is determined by the rank $r$ and degree $d$
    and we use the notation $U^{r,d}$ instead of $U^P$.
    
    If $\pi:D\to C$ is a Galois cover of a smooth projective curve $C$ of genus $\geq 2$, 
    then for $r\geq 1$ and $d\in \mathbb{Z}$ we have the following:
    \begin{enumerate}[resume*]
        \item $U^{r,d}\subseteq M^{s,r,d}_D$ is big.
        \item If $d$ and $\deg(\pi)$ are coprime, then $U^{r,d}=M^{s,r,d}_D$.
    \end{enumerate}
\end{lemma}

Note that Lemma \ref{lemma-construction-orbit} (i) is a generalization of \cite[Proposition 5.1]{bp}.

\begin{proof}
    Let $\sigma\in G$. To see that $U_{\sigma}$ is open, consider the pullback square
    \[
        \begin{tikzcd}
            Z_{\sigma} \ar[r] \ar[d] & M^{G-ss,P}_Y \ar[d,"(\sigma^{\ast}{,}\id)"]\\
            M^{G-ss,P}_Y \ar[r,"\Delta"] & M^{G-ss,P}_Y\times_k M^{G-ss,P}.
        \end{tikzcd}
    \]
    As $M^{G-ss,P}_Y$ is projective, $\Delta$ is a closed immersion and
    $Z_{\sigma}$ is closed. Clearly, $Z_{\sigma}$ 
    is the complement of $U_{\sigma}$ and we find the openness.

    Note that $W\in U^{P}$ iff $W\in M^{s,P}_Y$ and $\sigma^{\ast}W\ncong W$
    for $\sigma\in G\setminus\{e_G\}$ 
    as for stable vector bundles $S$-equivalence is the same as being isomorphic.
    
    (i): Let $W\in U^P$. Then $V:=\pi_{\ast}W$ is semistable of slope $\frac{\mu(W)}{\deg(\pi)}$ 
    by Lemma \ref{lemma-etale-pushforward}.
    Consider a stable saturated subsheaf $V'\subseteq V$. 
    Then $V'_{\mid Y}$ is a saturated subsheaf of $\pi^{\ast}\pi_{\ast}W\cong \bigoplus_{\sigma\in G}\sigma^{\ast}W$ 
    and thus reflexive.
    Furthermore, $V'_{\mid Y}$ is polystable of the same slope as $W$.
    As the graded object of the JH-filtration is unique,
    we find that $V'_{\mid Y}\cong \bigoplus_{\sigma \in \Sigma}\sigma^{\ast}W$ for some
    $\Sigma\subseteq G$. In particular, $V'$ is a stable vector bundle.
    By Lemma \ref{lemma-key}, $G$ acts transitively
    on the isomorphism classes of the $\sigma^{\ast}W,\sigma\in \Sigma$.
    We conclude $V'_{\mid Y} = V_{\mid Y}$ and thus $V'=V$
    as for $\sigma\in G\setminus\{e_G\}$ we have $\sigma^{\ast}W\ncong W$.

    (ii): Let $W\in M^{s,P}_Y$ such that $W\notin U^P$. 
    Then there exists $\sigma\in G\setminus\{e_G\}$ 
    such that $\sigma^{\ast}W\cong W$.
    Let $G'$ be the subgroup of $G$ generated by $\sigma$.
    Then $G'$ is cyclic and non-trivial. For cyclic covers the 
    notions of invariance and linearization agree, see \cite[Lemma 4.6]{fun}.
    Thus, $W$ descends to the intermediate cover $Y/G'$.

    Conversely, if $W$ descends to some intermediate cover $Y\to Y'$ 
    with Galois group $G'$ such that $Y\ncong Y'$,
    then there exists $W'$ on $Y'$ such that $W'_{\mid Y}\cong W$.
    Thus, $\sigma'^{\ast}W\cong W$ for $\sigma'\in G'$. This concludes (ii).

    (iii)-(iv):
    Let $r\geq 1, d\in\mathbb{Z}$ and $D\to C$ be a Galois cover 
    of a smooth projective curve $C$ of genus $\geq 2$.
    
    (iii): Then by the alternative description
    given in (ii), the open
    $U^{r,d}$ is obtained by removing vector bundles of rank $r$ that pullback from an
    intermediate cover, i.e., 
    \[
        U^{r,d}=\bigcap_{D\xrightarrow{\pi'} D'\to C}\big{(}M^{ss,r,d}_D\setminus\pi'^{\ast}(M^{ss,r,\frac{d}{\deg(\pi')}}_{D'})\big{)}\cap M^{s,r,d}_D,
    \]
    where the intersection is taken over intermediate covers $D'\to C$ such
    that $D\neq D'$. If $M^{ss,r,\frac{d}{\deg(\pi')}}_{D'}$ is non-empty, then it has dimension
    $\frac{r^2}{\deg(\pi')}(g_{D}-1)+1$
    by Riemann-Hurwitz. Thus,
    we find that $U^{r,d}$ is big.

    (iv): If $d$ and $\deg(\pi)$ are coprime, 
    then the moduli spaces $M^{ss,r,\frac{d}{\deg(\pi')}}_{D'}$ considered in (iii)
    are empty and we obtain (iv).
\end{proof}

\begin{lemma}
\label{lemma-pushforward-torsor}
    Let $\pi:D\to C$ be a Galois cover of smooth projective curves with Galois group $G$.
    Let $r\geq 1$ and $d$ be integers.
    For the open 
    \[
        U^{r,d}=\{V\in M^{s,r,d}_D \mid \sigma^{\ast}V\ncong V, \text{ for all }\sigma\in G\setminus\{e_G\}\} \subseteq M^{s,r,d}_D
    \]
    defined in Lemma \ref{lemma-construction-orbit}
    pushforward along $\pi$ induces a Cartesian diagram
    \[
        \begin{tikzcd}
            M^{ss,r,d}_D \ar[r,"\pi_{\ast}"] & M^{ss,\#(G)r,d}_C\\
            U^{r,d} \ar[u, open]\ar[r,"\pi_{\ast}"] & M^{s,\#(G)r,d}_C \ar[u, open].
        \end{tikzcd}
    \]
    In particular, $\pi_{\ast}:U^{r,d}\to M^{s,\#(G)r,d}_C$ is finite.
\end{lemma}

\begin{proof}
    By Lemma \ref{lemma-construction-orbit}, we have
    $U^{r,d}\subseteq \pi_{\ast}^{-1}(M^{s,\#(G)r,d}_C)$.
    For the other inclusion,
    let $V\in M^{ss,r,d}_D$ such that $\pi_{\ast}V$ is stable.
    Then $V$ is stable, as a subsheaf $W\subseteq V$ of the same slope
    yields a subsheaf $\pi_{\ast}(W)\subseteq \pi_{\ast}V$ of the same slope.
    By the description of $U^{r,d}$ given in Lemma \ref{lemma-construction-orbit} (ii),
    we find that it suffices to show that $V$ does not descend to an intermediate
    cover $D\to D'\to C$ such that $D'\ncong D$.

    Assume by way of contradiction that $V\cong V'_{\mid D}$ for some vector bundle $V'$ on $D'$, where
    $D\xrightarrow{\pi'} D'\to C$ is an intermediate cover such that $D\ncong D'$.
    Then $V'\subseteq \pi'_{\ast}V$ is a proper subsheaf of the same slope. Thus, $\pi'_{\ast}V$ is not stable
    and the same holds for $\pi_{\ast}V$.

    Finiteness is preserved under base change and we conclude by Lemma \ref{lemma-pushforward-finite}.
\end{proof}

    \section{Prime to p decomposition strata}

Let $C$ be a smooth projective curve of genus $g_C\geq 2$.
The cover $C_{r,good}\to C$ constructed in \cite[Theorem 1]{fun}
can be iterated to obtain a cover $C_{r,split}$ checking for the eventual
decomposition of a stable vector bundle of rank $r$ on $C$.
This gives rise to the prime to $p$ decomposition strata.
Similarly to the decomposition strata with respect to a cover,
the dimension of the prime to $p$ decomposition strata can be estimated.
The main difference is that these estimates are sharp
as long as the characteristic is avoided. To obtain sharp estimates we find a way to construct
vector bundles with prescribed prime to $p$ decomposition using cyclic covers.

\subsection{A split cover}

We begin with the definition of the cover $C_{r,split}$. 
This still works on normal projective varieties:
\begin{definition}
\label{definition-prime-to-p-strata}
    Let $X$ be a normal projective variety. Let $r\geq 2$. 
    We define a prime to $p$ Galois cover $\pi_{r,split}:X_{r,split}\to X$
    inductively as follows:
    Set $X_1$ as $X_{r,good}$.
    Then define  for $1\leq j< r-1$ the cover $X_{j+1}$ as a prime to $p$ Galois cover
    dominating $(X_j)_{l,good}$ for $l \leq r-j$ and $l\mid r$.
    We define $X_{r,split}:=X_{r-1}$.

    Let $V$ be a stable vector bundle on $X$ with Hilbert polynomial $P$.
    The \emph{prime to $p$ decomposition type} (or just \emph{decomposition type})
    of $V$ is the decomposition type with respect to the cover $X_{r,split}\to X$.

    We call the stratification of $M^{s,P}_X$ with respect to
    $\pi_{r,split}$ the \emph{prime to $p$ decomposition stratification}
    (or just \emph{decomposition stratification})
    and denote it for $m\mid r$ by
    \[
        Z^{s,P}(m):=Z^{s,P}(m,\pi_{r,split}).
    \]

    If $X$ is a smooth projective curve, the Hilbert polynomial 
    is determined by the rank $r$ and the degree $d$ and we use the notation
    $Z^{s,r,d}(m)$ instead.
    
    The notations $X_{r,split}$ and $Z^{s,P}(m)$ are justified 
    as the decomposition type stays the same on any
    prime to $p$ cover dominating $X_{r,split}$ as we show in the next lemma.
\end{definition}

\begin{lemma}
\label{lemma-prime-to-p-infinity}
    Let $X$ be a normal projective variety. Let $r\geq 2$.
    Then a stable vector bundle of rank $r$ on $X$
    decomposes on $X_{r,split}$ into a direct sum of prime to $p$ stable vector bundles.

    In particular, the decomposition types with respect to a prime to $p$ Galois cover
    $Y\to X$ dominating $X_{r,split}\to X$ and with respect to $X_{r,split}\to X$ agree.
\end{lemma}

\begin{proof}
    Let $V$ be stable vector bundle of rank $r$ on $X$.
    We follow the behaviour of $V_{\mid X_j}$ for $j=1,\dots ,r-1$
    of Definition \ref{definition-prime-to-p-strata}.
    Let $r_j$ be the decomposition type of $V$ with respect to $X_j\to X$.
    Then we have $r_{j+1}\leq r_j$ for $1\leq j<r-1$ and equality holds iff 
    $V_{\mid X_j}$ is a direct sum of stable vector bundles on $X_j$ that
    remain stable on $X_{j+1}$.
    By construction of $X_{j+1}$ this
    is the case iff $V_{\mid X_j}$ decomposes into 
    a direct sum of prime to $p$ stable vector bundles on $X_{j}$.

    As there are $r-1$ covers $X_{r-1}\to\dots \to X_1\to X$,
    we find $r_{r-1}=1$ or $r_{r-1}=r_j$ for some $j<r-1$.
    In both cases $V_{\mid X_{r-1}}=V_{\mid X_{r,split}}$ decomposes 
    into a direct sum of prime to $p$ stable vector bundles.

    If $Y\to X$ is a prime to $p$ Galois cover dominating $X_{r,split}\to X$,
    consider the decomposition of $V_{\mid X_{r,split}}\cong \bigoplus_{i=1}^{n} V_i^{\oplus e}$
    into pairwise non-isomorphic stable vector bundles of rank $m$.
    By the above discussion $V_i$ is prime to $p$ stable.
    Thus, $V_{\mid X}\cong \bigoplus_{i=1}^n (V_i)_{\mid Y}^{\oplus e}$ 
    is a decomposition into stable vector bundles of rank $m$.
\end{proof}

\begin{remark}
    Note that the refined decomposition type might still change and
    is thus not independent of the cover.  For example only finitely many prime to $p$ trivializable
    stable vector bundles of rank $r$ become trivial on $X_{r,split}$. However, 
    there are infinitely many prime to $p$ trivializable stable
    bundles of rank $r$ if the prime to $p$ completion of the \'etale fundamental group is large enough;
    e.g. if $X=C$ is a smooth projective curve of genus at least $2$.
\end{remark}

\subsection{Sharp dimension estimates}

We can now give the mostly sharp estimates of the decomposition strata.
This is a generalization of \cite[Theorem 2]{fun}.
\begin{theorem}
\label{theorem-dimension-estimate-sharp}
    Let $C$ be a smooth projective curve of genus $g_C\geq 2$.
    Let $r\geq 2$ and $d\in \mathbb{Z}$.
    Then for $m\mid r$ we have the following:
    \begin{enumerate}[(i)]
        \item $\dim(Z^{s,r,d}(m))\leq (\frac{r}{m})'m^2(g_C-1)+1$, where $(\frac{r}{m})'$ denotes the prime to $p$ part of $\frac{r}{m}$.
        \item If $p \nmid \frac{r}{m}$, then we have $\dim(Z^{s,r,d}(m))=rm(g_C-1)+1$.
    \end{enumerate}
\end{theorem}

\begin{proof}
    (i): The estimate follows immediately from applying Theorem \ref{theorem-dimension-estimate}
    to the prime to $p$ Galois cover $C_{r,split}\to C$.

    (ii): Let $n=\frac{r}{m}$ and assume that $n$ is prime to $p$.
    Then cyclic covers of degree $n$ correspond to line bundles of order $n$.
    Thus, there exist such cyclic covers. Let $\pi:D\to C$ be cyclic of degree $n$.
    Denote the Galois group by $G$.

    By Lemma \ref{lemma-construction-orbit} and Lemma \ref{lemma-pushforward-torsor}, 
    the locus $U\subseteq M^{s,m,d}_D$ of stable vector bundles $W$ on $D$
    such that $\pi_{\ast}W$ is stable
    is open and non-empty.
    Let $M^{p'-s,m,d}_D$ be the locus of prime to $p$ stable vector bundles on $D$.
    Recall that $M^{p'-s,m,d}_D$ is non-empty and open, see \cite[Theorem 2]{fun}.
    As $M^{s,m,d}_D$ is irreducible, the intersection
    $U':=U\cap M^{p'-s,m,d}_D$ is open and non-empty as well.

    We claim that for $W\in U'$ the direct image $\pi_{\ast}W$ has prime to $p$ 
    decomposition type $m$. Indeed, by affine base change we obtain
    $\pi^{\ast}\pi_{\ast}W\cong \bigoplus_{\sigma\in G}\sigma^{\ast}W$ 
    which is a direct sum of prime to $p$ stable vector bundles.

       Pushforward induces a finite morphism
    $U \xrightarrow{\pi_{\ast}} M^{s,r,d}_C$,
    see Lemma \ref{lemma-pushforward-torsor}.
    We obtain
    \[
        \dim(M^{ss,m,d}_D)= \dim(U') \leq \dim(Z^{s,r,d}(m)).
    \]
    Thus, we conclude 
    \[
        \dim(Z^{s,r,d}(m))=rm(g_C-1)+1
    \]
    by Riemann-Hurzwitz and (i).
\end{proof}

\begin{remark}
    Note that the dimension estimates can be far from being sharp if $p\mid r$:
    If $r=p^n$ for some $n\geq 1$ and $d$ is prime to $p$, then
    every semistable vector bundle of rank $r$ and degree $d$ is prime to $p$ stable.
    In particular, the decomposition strata $Z^{s,r,d}(m)$ are empty for $m<r$.
\end{remark}

\section{The closure of prime to p trivializable bundles}

As an application of the theory developed we study
the closure of the prime to $p$ trivializable stable vector bundles.

In positive characteristic the \'etale trivializable bundles are 
dense by a theorem due Ducrohet and Mehta, see \cite[Corollary 5.1]{dm}. 
This is no longer the case if we only consider 
prime to $p$ trivializable bundles as the general bundle 
remains stable on all prime to $p$ covers, see
\cite[Theorem 2]{fun}.
It is natural to ask the question what the closure $Z$ of 
the prime to $p$ trivializable bundles is.
This is closely related to the smallest prime to $p$ decomposition stratum
and we can estimate the dimension of $Z$. 
This estimate is sharp if the characteristic is avoided.

\begin{definition}
    Let $X$ be a normal projective variety. 
    We call a vector bundle $V$ of rank $r$
    on $X$ \emph{prime to $p$ trivializable} if 
    there exists a prime to $p$ cover $Y\to X$ such that 
    $V_{\mid Y}$ is the trivial bundle of rank $r$.
\end{definition}
Note that as every prime to $p$ cover is dominated by a prime to $p$ 
Galois cover, we could have alternatively required that
prime to $p$ trivializable vector bundles become trivial on a prime to $p$
Galois cover.

As a direct consequence of \cite[1.2 Proposition]{ls}, 
representations of the prime to $p$ completion of the \'etale fundamental group
correspond to prime to $p$ trivializable vector bundles.

Also recall that in rank $1$ the prime to $p$ trivializable line bundles are dense
in $\Pic^0_C$ if $C$ is a smooth projective curve.

\begin{theorem}
\label{theorem-closure}
    Let $C$ be a smooth projective curve of genus $g_C\geq 2$. Let $r\geq 2$.
    Let $Z^{s,r}$ be the closure of the prime to $p$ 
    trivializable stable vector bundles in $M^{s,r,0}_C$ and $Z^{ss,r}$ be the closure of the prime to $p$ 
    trivializable vector bundles in $M^{ss,r,0}_C$.
    Then we have the following:
    \begin{enumerate}[(i)]
        \item $Z^{s,r}\subseteq Z^{s,r,0}(1)$.
        \item $\dim(Z^{s,r})\leq r'(g_C-1)+1$,
            where $r'$ is the prime to $p$ part of $r$.
        \item If $p\nmid r$, then $\dim(Z^{s,r})= \dim(Z^{s,r,0}(1))=r(g_C-1)+1$.
        \item $\dim(Z^{ss,r})=rg_C$.
    \end{enumerate}
\end{theorem}
   
\begin{proof}
    (i): The prime to $p$ decomposition type of a prime to $p$ trivializable bundle is $1$.
    Thus, we obtain $Z^{s,r}\subseteq Z^{s,r,0}(1)$ using that $Z^{s,r,0}(1)$
    is closed in $M^{s,r,0}_C$, see Lemma \ref{lemma-strata-closed}. 

    (ii): This is a direct consequence of (i) and the dimension estimate for the prime to $p$ decomposition strata,
    see Lemma \ref{theorem-dimension-estimate} (ii).

    (iii): Assume that $p\nmid r$. Let $\pi:D\to C$ be a cyclic cover of order $r$ and 
    Galois group $G\cong \mathbf{Z}/r\mathbf{Z}$.
    Consider the dense open subset $U\subseteq \Pic^0_D$ defined as
    \[
        U:=\{ L\in \Pic^0_D \mid \sigma^{\ast}L\ncong L, e_G\neq \sigma\in G\},
    \]
    see Lemma \ref{lemma-construction-orbit}.
    By Lemma \ref{lemma-pushforward-torsor}, pushforward induces 
    a finite morphism
    \[
        \pi_{\ast}:U\to M^{s,r,0}_C, L\mapsto \pi_{\ast}L.
    \]
    
    As the prime to $p$ trivializable line bundles are dense in $\Pic^0_D$
    and $\Pic^0_D$ is irreducible, we find that the prime to $p$ trivializable 
    line bundles contained in $U$ are still dense in $\Pic^0_D$.

    Note that for a prime to $p$ trivializable line bundle $L$ its conjugates $\sigma^{\ast}L,\sigma\in G$,
    are prime to $p$ trivializable as well. Thus, the pushforward $\pi_{\ast}L$
    is prime to $p$ trivializable as $\pi^{\ast}\pi_{\ast}L\cong \bigoplus_{\sigma\in G}\sigma^{\ast}L$.
    Therefore, $\pi_{\ast}$ factors as $U\to Z^{s,r}\to M^{s,r,0}_C$.

    The image $\pi_{\ast}(U)$ in $Z^{s,r}$ has dimension $\dim(U)=g_{D}$ 
    since $\pi_{\ast}$ is finite.
    By Riemann-Hurwitz and (ii) we conclude 
    \[
        r(g_C-1)+1\leq \dim(Z^{s,r})\leq \dim(Z^{s,r,0}(1))\leq r(g_C-1)+1.
    \]
    
    (iv): Observe that $Z^{ss,r}$ contains the image of the finite morphism
    \[
        \prod_{i=1}^r \Pic^0_C \to M^{ss,r,0}_C, (L_1,\dots,L_r)\mapsto \bigoplus_{i=1}^r L_i.
    \]
    Thus, we have $\dim(Z^{ss,r})\geq rg_C$.
    
    To obtain the other inequality, we claim that 
    $Z^{ss,r}$ is the union of the images of
    \[
        \varphi_{r_1,\dots,r_l}:\prod_{i=1}^{l} \overline{Z^{s,r_i}} \to M^{ss,r,0}_C, (V_1,\dots, V_l)\mapsto \bigoplus_{i=1}^l V_i
    \]
    for all possible ways to write $r=\sum_{i=1}^l r_i$ and the closure of $Z^{s,r_i}$ is taken in $M^{ss,r_i,0}_C$. 
    Indeed, the image of $\varphi_{r_1,\dots,r_l}$ is closed and 
    so is $\bigcup_{\sum_{i=1}^l r_i=r}\im(\varphi_{r_1,\dots,r_l})$.
    Furthermore, the union contains all prime to $p$ trivializable bundles of rank $r$.

    By (ii) we have 
    $\dim(\varphi_{r_1,\dots,r_l})\leq r(g_C-1)+l\leq rg_C$.
\end{proof}

\bibliographystyle{plain}

\bibliography{bibliography}

\end{document}